\pdfoutput=1
\RequirePackage{ifpdf}
\ifpdf 
\documentclass[pdftex]{sigma}
\else
\documentclass{sigma}
\fi

\numberwithin{equation}{section}

\newtheorem{Theorem}{Theorem}[section]
\newtheorem*{Theorem*}{Theorem}
\newtheorem{Corollary}[Theorem]{Corollary}
\newtheorem{Lemma}[Theorem]{Lemma}
\newtheorem{Proposition}[Theorem]{Proposition}

\theoremstyle{definition}
\newtheorem{Definition}[Theorem]{Definition}

\newtheorem{Remark}[Theorem]{Remark}

\usepackage{mathtools}

\usepackage{enumerate}

\newcommand{\cyclic}[1]{\stackrel{\text{\scriptsize $#1$}}{\mathfrak{S}}}
\newcommand{\R}{\mathbb{R}}

\renewcommand{\H}{\mathcal{H}}

\newcommand{\End}{\ensuremath{\mathrm{End}}}

\newcommand{\dd}{\mathrm{d}}
\newcommand\extalg{%
 \newlength{\len}%
 \settoheight{\len}{V}%
 \mathbin{%
 \resizebox{0.93\len}{0.93\len}{$\wedge$}%
 \kern-0.1em%
 }}%
\newcommand{\intprod}{\mathbin{\hbox to 0.7ex{%
 \kern-0.3ex
 \vrule height0.0777ex width0.971ex depth0ex
 \kern-0.055ex
 \vrule height1.165ex width0.0777ex depth0ex\hss}}%
}%

\begin{document}

\allowdisplaybreaks

\newcommand{\arXivNumber}{2602.01458}

\renewcommand{\PaperNumber}{104}

\FirstPageHeading

\ShortArticleName{Reducible Holonomy in Closed Torsion Geometries}

\ArticleName{Reducible Holonomy in Closed Torsion Geometries}

\Author{Leander STECKER}

\AuthorNameForHeading{L.~Stecker}

\Address{Universit\"at Leipzig, Augustusplatz 10, 04109 Leipzig, Germany}
\Email{\mail{leander.stecker@uni-leipzig.de}, \mail{l.stecker@arcor.de}}

\ArticleDates{Received May 01, 2026, in final form October 12, 2026; Published online October 29, 2026}

\Abstract{The purpose of this note is to show that a connection with closed skew-symmetric torsion and reducible holonomy admits a locally defined Riemannian submersion together with a projected geometry on the base. We reframe known submersion results for non-K\"ahler Bismut--Hermite--Einstein manifolds and sHKT structures in this context. For homogeneous SKT structures on semi-simple Lie groups, we obtain the holonomy decomposition leading to holomorphic submersions over generalized flag manifolds.}

\Keywords{closed torsion connection; holonomy; Riemannian submersions; SKT manifolds; homogeneous SKT spaces}

\Classification{53B05; 53C29; 53B35; 32M10}

\section{Introduction}

Geometries with torsion generalize special holonomy manifolds. Where the latter are often too rigid, adapted connections with torsion can expand the scope while retaining desired properties. In particular, geometries with skew-symmetric and parallel or closed torsion have shown great promise. Parallel torsion is considered more convenient for its geometric properties, see, for example, \cite{AFS, FriG2, FrIv}, yet closed torsion geometries appear canonically in many circumstances~\mbox{\cite{FLMM_G2, MGFJS_pluri}}. Applying tools from parallel torsion to closed torsion has seen efforts to combine both types~\cite{BFG_BTP, ZZ_BTP}. However, this comes with its own share of restrictions, namely a~very narrow class of possible manifolds, see \cite[Theorem~4.1]{AFF}. In this work, we instead rediscover a~dimensional reduction argument from parallel torsion, see~\cite{CMS}, in the closed torsion world.

The fundamental principle goes back to de~Rham's theorem arguing that a reducible holonomy representation should introduce a split geometry. Indeed, our main theorem gives the following interpretation.

\begin{Theorem}\label{thm1}
Suppose $\nabla$ is a metric connection with closed skew torsion $T$ and $TM=\mathcal{H}\oplus\mathcal{V}$ splits orthogonally as a representation of the holonomy group $\mathrm{Hol}(\nabla)$. Assume further that
\begin{align*}
T=T^{\mathcal{H}}+T^{m}+T^{\mathcal{V}}\in \Lambda^3\mathcal{H}\oplus\Lambda^2\mathcal{H}\!\wedge\!\mathcal{V}\oplus \Lambda^3\mathcal{V}\subsetneq \Lambda^3TM.
\end{align*}
Then
\begin{enumerate}[$(a)$]\itemsep=0pt
\item there exists a locally defined Riemannian submersion $\pi\colon (M,g)\to (N,g_N)$ with totally geodesic fibers tangent to $\mathcal{V}$,
\item the purely horizontal part of the connection $T^\mathcal{H}$ is projectable, $\pi^*\check{T}\coloneqq T^\mathcal{H}$,
\item and $\check{\nabla}=\nabla^{g_N}+\frac 12 \check{T}$ is a connection with, not necessarily closed, skew torsion on $N$ satisfying
$\check{\nabla}_XY=\pi_*\bigl(\nabla_{\overline{X}}\overline{Y}\bigr)$,
where $\overline{X}$, $\overline{Y}$ denote the horizontal lifts of $X,Y\in TN$.
\end{enumerate}
\end{Theorem}

While we only obtain a locally defined Riemannian submersion instead of a full product geometry, it allows for more intricate relations between the total, base and fiber spaces. In the parallel case, this gave rise to classification results \cite{ADS21, MorSch} and new solutions to geometric systems of equations~\cite{ADS23, G2heterotic}.

Our focus here will be employing the submersion theorem to SKT manifolds. These manifolds are a class of non-K\"ahler Hermitian manifolds that rose to prominence due to their connection with the generalized Ricci and pluriclosed flow, \cite{MGF}, and being one of three special non-K\"ahler Hermitian geometries in the Fino--Vezzoni conjecture \cite{FGsurvey,FiVe, OVFinoVezConj}. Prime examples of compact non-K\"ahler homogeneous SKT manifolds are Bismut-flat structures on compact semi-simple Lie groups of even dimension with their bi-invariant metrics. Recent results in \cite{FGSKT,LauMon, Pham} have shown that there is a vivid collection of left-invariant metrics generalizing those bi-invariant examples. Our second main Theorem~\ref{SamelsonHolonomy} computes the decomposition of their Bismut-holonomy. To our knowledge, these are the first non-trivial examples of SKT manifolds where the Bismut holonomy decomposition is fully determined. We find that
\begin{align*}
\mathfrak{g}=\mathfrak{t}\oplus\bigoplus_{\alpha\in \Delta_{I_{\max}}}\mathfrak{g}_{\alpha}^\mathbb{R}\oplus \Biggl(\sum_{\alpha\in\Delta^+\setminus\Delta_{I_{\max}}}\mathfrak{g}_\alpha^\mathbb{R}\Biggr)
\end{align*}
for a span of simple roots $\Delta_{I_{\max}}$ directly determined by the metric. Therefore, we obtain Hermitian submersions $G\to G^\mathbb{C}/P$ for any generalized flag $G^\mathbb{C}/P$ as special examples of Theorem~\ref{thm1}.

The paper is organized as follows. Section~\ref{Chp2} is devoted to the main Theorem~\ref{submthm}. We~recall the necessary preliminaries on connections with closed skew torsion in order to prove the submersion theorem. In Theorem~\ref{holthm}, we show under which circumstances the holonomy projects to the base. Our theorem also covers the parallel torsion case where this question has been shown to be delicate in some situations. We investigate the situation for SKT manifolds in Section~\ref{Chp3} reframing in Theorems~\ref{thmBHE} and~\ref{thmsHKT} recently proved submersion theorems for BHE and sHKT manifolds as special cases of our main theorem. The final section is all about homogeneous SKT structures on semisimple Lie groups with Samelson complex structures. We give a full decomposition of the Bismut holonomy in Theorem~\ref{SamelsonHolonomy}, proving that they admit holonomy induced Hermitian submersions to all generalized flag manifolds.

\section{Submersion theorem}\label{Chp2}

\begin{Definition}
A connection $\nabla$ is said to have skew-torsion if the torsion
\begin{align*}
T(X,Y,Z)\coloneqq g(\nabla_XY-\nabla_YX-[X,Y],Z)\in \Lambda^3T^*M.
\end{align*}
In this case, $\nabla=\nabla^g+\frac 12 T$.
\end{Definition}

Here and in the following we will always refer to the Levi-Civita connection of $(M,g)$ and its associated objects by a superscript $g$.

One interpretation of torsion is a failure of $\nabla$ to compute the exterior derivative. One gets the amended formula for a $k$-form $\omega$
\begin{align}
\mathrm{d}\omega(X_0,\dots,X_k)&=\sum_{i=0}^k(-1)^i(\nabla_{X_i}\omega)\bigl(X_0,\dots,\hat{X_i},\dots,X_k\bigr)\nonumber\\
&\quad-\sum_{i<j}(-1)^{i+j}\omega\bigl(T(X_i,X_j),X_0,\dots,\hat{X_i},\dots,\hat{X_j},\dots,X_k\bigr).\label{exterior}
\end{align}

For a connection $\nabla$ with skew-torsion, we consider its curvature
\begin{align*}
R^\nabla(X,Y,Z,V)\coloneqq g\bigl(R^\nabla(X,Y)Z,V\bigr)\coloneqq g(([\nabla_X,\nabla_Y]-\nabla_{[X,Y]})Z,V).
\end{align*}
Note that we cannot assume the usual symmetries from the Levi-Civita curvature. However, we obtain the following Bianchi identity
\begin{align}\label{Bianchi}
\cyclic{X,Y,Z}\!\!R(X,Y,Z,V)=\mathrm{d}T(X,Y,Z,V)-\sigma_T(X,Y,Z,V)+\nabla_VT(X,Y,Z),
\end{align}
where on the left-hand side \smash{$\cyclic{X,Y,Z}$} is short for a sum over all cyclic permutations of $X$, $Y$, $Z$. On the right-hand side we denote by $\sigma_T$ the $4$-form
\begin{align*}
\sigma_T(X,Y,Z,V)=\cyclic{X,Y,Z}\!\! g\bigl(T(X,Y)^{\#},T(Z,V)^{\#}\bigr)\in \Lambda^4T^*M.
\end{align*}

\begin{Definition}
We call $(M,g,\nabla)$ a closed torsion geometry if $\nabla$ has skew-torsion $T$ satisfying $\mathrm{d}T=0$.
\end{Definition}

We can immediately state our main theorem.

\begin{Theorem}\label{submthm}
Suppose $(M,g,\nabla)$ is a closed torsion geometry with torsion $T$ and $TM=\mathcal{H}\oplus\mathcal{V}$ splits orthogonally as a representation of the holonomy group $\mathrm{Hol}(\nabla)$. Assume further that
\begin{align}\label{projecttau}
T=T^{\mathcal{H}}+T^{m}+T^{\mathcal{V}}\in \Lambda^3\mathcal{H}\oplus\Lambda^2\mathcal{H}\!\wedge\!\mathcal{V}\oplus \Lambda^3\mathcal{V}\subsetneq \Lambda^3TM.
\end{align}
Then
\begin{enumerate}[$(a)$]\itemsep=0pt
\item there exists a locally defined Riemannian submersion $\pi\colon (M,g)\to (N,g_N)$ with totally geodesic fibers tangent to $\mathcal{V}$,
\item the purely horizontal part of the connection $T^\mathcal{H}$ is projectable, $\pi^*\check{T}\coloneqq T^\mathcal{H}$,
\item and $\check{\nabla}=\nabla^{g_N}+\frac 12 \check{T}$ is a connection with, not necessarily closed, skew torsion on $N$ satisfying
\begin{align}\label{pinabla}
\check{\nabla}_XY=\pi_*\bigl(\nabla_{\overline{X}}\overline{Y}\bigr),
\end{align}
where $\overline{X}$, $\overline{Y}$ denote the horizontal lifts of $X,Y\in TN$.
\end{enumerate}
\end{Theorem}

\begin{Remark}\label{RemIntegrable}
The condition \eqref{projecttau} is equivalent to $\mathcal{V}$ being integrable. Indeed, for $V,W\in\mathcal{V}$ and $X\in \mathcal{H}$
\begin{align*}
g([V,W],X)=g(\nabla_VW,X)-g(\nabla_WV,X)-T(V,W,X)
\end{align*}
and, hence $[\mathcal{V},\mathcal{V}]\subset\mathcal{V}$ if and only if $T$ vanishes on $\Lambda^2\mathcal{V}\otimes \mathcal{H}$.

Hence, property \eqref{projecttau} immediately induces a foliation or, equivalently, a locally defined submersion. These local submersion patches will combine to a global submersion only for regular foliations. Staying in the terminology of foliations, Theorem~\ref{submthm} can be rephrased in terms of a~transverse geometry. However, for local results as in the following we may just work with $\pi$ as a submersion.
\end{Remark}

\begin{proof}[Proof of Theorem~\ref{submthm}]
We note that by \eqref{projecttau} and the invariance of $\mathcal{V}$ under $\nabla$ for any vertical vector fields $V,W\in \mathcal{V}$, we have
\begin{align*}
\nabla^g_VW=\nabla_VW-\tfrac 12T(V,W)\in\mathcal{V}.
\end{align*}
Therefore, the distribution $\mathcal{V}$ is integrable and a curve is geodesic on the integral submanifold tangent to $\mathcal{V}$ if and only if it is a geodesic in $M$. The integral submanifolds give rise to a foliation and hence to a submersion $\pi$ from a small neighborhood $U\subset M$ to a local transverse section~$S$. We show that the metric restricted to $\mathcal{H}\times\mathcal{H}$ is constant along vertical vector fields and therefore projectable. For $X,Y\in\H$, we have{\samepage
\begin{align*}
\begin{split}
(L_Vg)(X,Y)&=V(g(X,Y))-g([V,X],Y)-g(X,[V,Y])\\
&=g(\nabla^g_XV,Y)+g(X,\nabla^g_YV)
=g(\nabla_XV,Y)+g(X,\nabla_YV)=0
\end{split}
\end{align*}
since $\mathcal{V}$ is preserved by $\nabla$. This proves that $\pi$ is a Riemannian submersion.}

To prove the second assertion, we denote $T^\mathcal{H}=\mathrm{pr}_{\Lambda^3\mathcal{H}}T$. If we show that $T^\mathcal{H}$ is constant along the fibers, it projects to a well-defined $3$-form $\check{T}$. Let $V\in \mathcal{V}$. By Cartan's formula, $L_VT^{\mathcal{H}}=\iota_VT^{\mathcal{H}}$ as $\iota_VT^{\mathcal{H}}=0$. Then by \eqref{exterior}, we obtain
\begin{align}
L_VT^\mathcal{H}(X,Y,Z)&=\mathrm{d}T^\mathcal{H}(V,X,Y,Z)\notag\\
&=\bigl(\nabla^T_VT^\mathcal{H}\bigr)(X,Y,Z)+\cyclic{X,Y,Z}\!\!T^\mathcal{H}(T(V,X),Y,Z).\label{Torsionintorsion}
\end{align}
Whenever either $X,Y,Z\in \mathcal{V}$ this vanishes. Indeed, by \eqref{projecttau}, we have that $T(V,Z)\in\mathcal{V}$ if $Z\in \mathcal{V}$ and, hence, $T(V,Z)\intprod T^\mathcal{H}=0$. Now consider $X,Y,Z\in\mathcal{H}$. Since $\nabla$ preserves $\mathcal{V}$ and~$\mathcal{H}$, the curvature $R^{\nabla^T}(X,Y,Z,V)=0$ and the Bianchi identity for connections with skew torsion~\eqref{Bianchi} implies
\begin{align*}
0&=\cyclic{X,Y,Z}\!\! R^{\nabla}(X,Y,Z,V)=\mathrm{d}T(X,Y,Z,V)-\sigma_T(X,Y,Z,V)+(\nabla_VT)(X,Y,Z)\\
&=\cyclic{X,Y,Z}\!\!g(T(V,X),T(Y,Z))+(\nabla_VT)(X,Y,Z)\\
&=\cyclic{X,Y,Z}\!\!T^\mathcal{H}(T(V,X),Y,Z)+(\nabla_VT)(X,Y,Z)=L_VT^\mathcal{H}(X,Y,Z),
\end{align*}
where we used again that $T(V,Z)\in \mathcal{H}$ and the last step is~\eqref{Torsionintorsion}.

Equation \eqref{pinabla} follows directly from $\nabla^{g_N}_{X}{Y}=\pi_*\bigl(\nabla^{g}_{\overline{X}}\overline{Y}\bigr)$ for Riemannian submersions.
\end{proof}

\begin{Corollary}
In the situation of Theorem~{\rm \ref{submthm}}, the base space $\bigl(N, g_N,\check{\nabla}\bigr)$ is again a closed torsion geometry if and only if the torsion $\check{T}$ satisfies
\begin{align}\label{dcT}
\mathrm{d}\check{T}(X_1,X_2,X_3,X_4)=-2\sigma_{T^m}\bigl(\overline{X_1},\overline{X_2},\overline{X_3},\overline{X_4}\bigr)=0.
\end{align}
\end{Corollary}

\begin{proof}
Denote $T^\perp=T^m+T^\mathcal{V}\in \Lambda^2\mathcal{H}\!\wedge\!\mathcal{V}\oplus\Lambda^3\mathcal{V}$, then
\begin{align*}
\mathrm{d}\check{T}(X_1,X_2,X_3,X_4)&=\mathrm{d}T^\mathcal{H} \bigl(\overline{X_1},\overline{X_2},\overline{X_3},\overline{X_4}\bigr)=-\dd T^\perp\bigl(\overline{X_1},\overline{X_2},\overline{X_3},\overline{X_4}\bigr)\\
&=\bigl(\nabla_{\overline{X_4}}^TT^\perp\bigr) \bigl(\overline{X_1},\overline{X_2},\overline{X_3}\bigr) -\!\!\cyclic{X_1,X_2,X_3}\!\!\bigl(\nabla^T_{\overline{X_1}}T^\perp\bigr) \bigl(\overline{X_2},\overline{X_3},\overline{X_4}\bigr)\\
&\quad - \cyclic{X_1,X_2,X_3}\!\!T^\perp\bigl(T\bigl(\overline{X_1},\overline{X_2}\bigr),\overline{X_3},\overline{X_4}\bigr)\\
& \quad - \cyclic{X_1,X_2,X_3}\!\!T^\perp\bigl(\overline{X_1},\overline{X_2},T\bigl(\overline{X_3},\overline{X_4}\bigr)\bigr).
\end{align*}
Since $\nabla$ preserves the type $\Lambda^2\mathcal{H}\wedge \mathcal{V}\oplus\Lambda^3\mathcal{V}$ of $T^\perp$, both covariant derivative terms vanish. Furthermore, on horizontal vectors $\overline{X_1}, \overline{X_2}\in \mathcal{H}$, we have $T^\perp\bigl(\overline{X_1},\overline{X_2},Z\bigr)=T^\perp\bigl(Z,\overline{X_1},\overline{X_2}\bigr) =g|_{\mathcal{V}\times\mathcal{V}}\bigl(T^m(\overline{X_1},\overline{X_2}),Z\bigr)$ and hence
\begin{align*}
\mathrm{d}\check{T}(X_1,X_2,X_3,X_4)&=-\!\!\cyclic{X_1,X_2,X_3}\!\!T^\perp \bigl(T\bigl(\overline{X_1},\overline{X_2}\bigr),\overline{X_3},\overline{X_4}\bigr) -\!\!\cyclic{X_1,X_2,X_3}\!\!T^\perp\bigl(\overline{X_1},\overline{X_2},T\bigl(\overline{X_3},\overline{X_4}\bigr)\bigr)\!\!\\
&=-2\!\!\cyclic{X_1,X_2,X_3}\!\!g\bigl(T^m\bigl(\overline{X_1},\overline{X_2}\bigr), T^m\bigl(\overline{X_3},\overline{X_4}\bigr)\bigr).\tag*{\qed}
\end{align*}\renewcommand{\qed}{}
\end{proof}

\begin{Corollary}
Let $(M,g,\nabla)$ be a geometry with closed torsion $T$ and $TM=V_1\oplus V_2$ an orthogonal holonomy-invariant decomposition. Then locally
\begin{align*}
(M,g,\nabla)=\bigl(M_1,g_1,\nabla^{T_1}\bigr)\times \bigl(M_2,g_2,\nabla^{T_2}\bigr)
\end{align*}
into closed torsion geometries if and only if the torsion is decomposable, i.e., $T=T_1+T_2$ with $T_i\in\Lambda^3V_i$.
\end{Corollary}

\begin{proof}
If the torsion is decomposable, \eqref{projecttau} is satisfied for $\mathcal{V}=V_1$, $\mathcal{H}=V_2$ and vice versa. Hence, both distributions are integrable and we obtain locally defined Riemannian projection maps to their respective integral submanifolds. As the mixed torsion parts vanish, we have \eqref{dcT} satisfied for both $T_1$ and $T_2$. The converse is clear.
\end{proof}

Note that the submersion and hence all above results are local.
Exploiting the splitting $TM=\mathcal{V}\oplus\mathcal{H}$ and \eqref{pinabla}, we collect some remarkable identities for the curvature $R$ of $\nabla$.

\begin{Proposition}
Let $\pi\colon M\to N$ be a Riemannian submersion, $\nabla$ and $\check{\nabla}$ connections on $M$ and $N$ as above.
\begin{enumerate}[$(a)$]\itemsep=0pt
\item Then $g(\nabla_XY,Z)=T(X,Y,Z)$ for any vector $X\in\mathcal{V}$, $Z\in\mathcal{H}$ and basic vector field $Y\in \Gamma\mathcal{H}$. In particular, both sides are tensorial.
\item The curvatures $R$ of $\nabla$ and $\check{R}$ of $\check{\nabla}$ are related by
\begin{align*}
R(X,Y,Z,V)=\check{R}(\pi_*X,\pi_*Y,\pi_*Z,\pi_*V)-g(T^m(X,Y),T^m(Z,V)),
\end{align*}
for any $X,Y,Z,V\in \mathcal{H}$.
\item For any $X,Y\in \mathcal{H}$ and $Z,V\in\mathcal{V}$, $R(X,Y,Z,V)=0$.
\end{enumerate}
\end{Proposition}

\begin{proof}
For the first identity, recall that for projectable vector fields $\pi_*[X,Y]=[\pi_*X,\pi_*Y]$ and hence, $[X,Y]\in\mathcal{V}$. Thus,
\begin{align*}
g(\nabla_XY,Z)=g(\nabla_YX+[X,Y],Z)+T(X,Y,Z)=T(X,Y,Z).
\end{align*}
As the identity is entirely tensorial, we may extend all vector fields as horizontal lifts. Then, by~\eqref{pinabla},
\begin{align*}
\check{R}(\pi_*X,\pi_*Y,\pi_*Z,\pi_*V)&= g_N\bigl(\bigl(\check{\nabla}_{\pi_*X}\check{\nabla}_{\pi_*Y} -\check{\nabla}_{\pi_*Y}\check{\nabla}_{\pi_*X}-\check{\nabla}_{[\pi_*X,\pi_*Y]}\bigr)\pi_*Z,\pi_*V\bigr)\\
&=g((\nabla_X\nabla_Y-\nabla_Y\nabla_X-\nabla_{[X,Y]-[X,Y]_\mathcal{V}})Z,V)\\
&=R(X,Y,Z,V)-g(\nabla_{[X,Y]_\mathcal{V}}Z,V),
\end{align*}
where $[X,Y]_\mathcal{V}=[X,Y]-\overline{\pi_*[X,Y]}=[X,Y]-\overline{[\pi_*X,\pi_*Y]}\in\mathcal{V}$ is the vertical part of the commutator. Further, for any $Z\in\mathcal{V}$, we get $g([X,Y]_\mathcal{V},Z)=g([X,Y],Z)=g([X,Y]-\nabla_YX-\nabla_XY)=T(X,Y,Z)$. Now using part~(a)
\begin{align*}
g(\nabla_{[X,Y]_\mathcal{V}}Z,V)&=T([X,Y]_\mathcal{V},Z,V)=g^\mathcal{V}(T(X,Y),T(Z,V)).
\end{align*}
Part~(c) uses an inverted Bianchi identity for skew-torsion due to \cite[Proposition~2.1]{IvanovStan}
\begin{align*}
\cyclic{Y,Z,V}\!\!R(X,Y,Z,V)=-\frac 12\mathrm{d}T(X,Y,Z,V)+(\nabla_X T)(Y,Z,V).
\end{align*}
Observe that $R(X,Z,V,Y)=0=R(X,V,Y,Z)$ as the two back entries are in different components of $TM=\mathcal{V}\oplus\mathcal{H}$. Hence, only $R(X,Y,Z,V)$ remains on the left-hand side. On the right-hand side, $\mathrm{d}T=0$ and $\nabla$ preserves the splitting of~$\Lambda^3TM$. Thus, the component~$\nabla_XT$ vanishes on vectors in $\Lambda^2\mathcal{V}\otimes\mathcal{H}$.
\end{proof}

\begin{Theorem}\label{holthm}
Let $\pi\colon M\to N$ be a Riemannian submersion, $\nabla$ and $\check{\nabla}$ connections on~$M$ and~$N$, respectively, such that $\check{\nabla}_XY=\pi_*\bigl(\nabla_{\overline{X}}\overline{Y}\bigr)$. Suppose $\Omega$ is a $\nabla$-parallel tensor field that projects under $\pi$ to a tensor field $\check{\Omega}$ on $N$. Then $\check{\nabla}\check{\Omega}=0$.

In particular, if $\mathrm{Hol}_0(\nabla)\subset\mathrm{Stab}(\Omega)$, then $\mathrm{Hol}_0(\check{\nabla})\subset\mathrm{Stab}\bigl(\check{\Omega}\bigr)$.
\end{Theorem}

\begin{Remark}
Here we consider $\check{\Omega}$ a projection of an $(s,t)$-tensor $\Omega$ if
\begin{align*}
g\bigl(\Omega\bigl(\overline{X}_1,\dots,\overline{X}_s\bigr),\overline{Y}_1\otimes\dots\otimes \overline{Y}_t\bigr)=\check{g}\bigl(\check\Omega(X_1,\dots,X_s),Y_1\otimes\dots\otimes Y_t\bigr)
\end{align*}
for all $X_1,\dots, X_s,Y_1,\dots,Y_t\in T_xN$. Such a~$\check{\Omega}$ exists if and only if
\begin{gather*}
\mathcal{L}_V\Omega=0,\qquad\text{and}\qquad
g(\Omega(X_1,\dots,V,\dots, X_s),Y_1\otimes\dots\otimes Y_t)=0
\end{gather*}
for all $V\in \mathcal{V}$ and $X_1,\dots,X_s,Y_1,\dots,Y_t\in \mathcal{H}$.
\end{Remark}

\begin{proof}[Proof of Theorem~\ref{holthm}]
Using that $\check{\Omega}$ is the projection of $\Omega$, we have
\begin{align*}
& \check{g}\bigl(\bigl(\check{\nabla}_Z\check{\Omega}\bigr)(X_1,\dots,X_s), Y_1\otimes \dots\otimes Y_t\bigr)
=\overline{Z}\bigl(g\bigl(\Omega\bigl(\overline{X}_1,\dots,\overline{X}_s\bigr),\overline{Y}_1\otimes\dots\otimes \overline{Y}_t\bigr)\bigr)\\
&\qquad-\sum_{i=1}^sg\bigl(\Omega\bigl(\overline{X}_1,\dots,\overline{\check\nabla_{Z}X_i},\dots,\overline{X}_s\bigr), \overline{Y}_1\otimes \dots\otimes \overline{Y}_t\bigr)\\
&\qquad-\sum_{i=1}^sg\bigl(\Omega\bigl(\overline{X}_1,\dots,\overline{X}_s\bigr), \overline{Y}_1\otimes\dots\otimes\overline{\check\nabla_ZY_j} \dots\otimes \overline{Y}_t\bigr).
\end{align*}
Now, by assumption, $\pi_*\bigl(\nabla_{\overline{X}}\overline{Y}\bigr)=\check\nabla_XY$ and, hence, $\nabla_{\overline{X}}\overline{Y}=\overline{\check{\nabla}_XY}+\bigl(\nabla_{\overline{X}}\overline{Y}\bigr)^\mathcal{V}$. The vertical parts do not contribute as $\Omega$ vanishes on such vectors and we are left with
\begin{align*}
&\overline{Z}\bigl(g\bigl(\Omega(\overline{X}_1,\dots,\overline{X}_s),\overline{Y}_1\otimes\dots\otimes \overline{Y}_t\bigr)\bigr)-\sum_{i=1}^sg\bigl(\Omega(\overline{X}_1,\dots,\nabla_{\overline{Z}} \overline{X}_i,\dots,\overline{X}_s),\overline{Y}_1\otimes \dots\otimes \overline{Y}_t\bigr) \\
&\quad\qquad-\sum_{i=1}^sg\bigl(\Omega\bigl(\overline{X}_1,\dots,\overline{X}_s\bigr), \overline{Y}_1\otimes\dots\otimes\nabla_{\overline{Z}}\overline{Y}_j \dots\otimes \overline{Y}_t\bigr) \\
&\qquad =g\bigl(\bigl(\nabla_{\overline{Z}} \Omega\bigr)\bigl(\overline{X}_1,\dots,\overline{X}_s\bigr), \overline{Y}_1\otimes\dots\otimes\overline{Y}_t\bigr)=0.\tag*{\qed}
\end{align*}\renewcommand{\qed}{}
\end{proof}

\begin{Remark}
The theorem above is not restricted to the case of Theorem~\ref{submthm}. In fact, for any Riemannian submersion $\pi\colon (M,g)\to (N,g_N)$ the Levi-Civita connections project
$\nabla^{g_N}_XY=\nabla^g_{\overline{X}}\overline{Y}$.
Additionally, it applies to the submersion theorem for parallel torsion, compare the version in~\cite{nKpaper}.
\end{Remark}

\begin{Remark}
The condition that $\Omega$ projects is crucial here. In~\cite{ADS21}, we have proved that even though the canonical connection $\nabla$ on a parallel $3$-$(\alpha,\delta)$-Sasaki manifold has $\mathrm{Hol}_0(\nabla)=\mathrm{Sp}(n)$, the holonomy on the quaternionic base is only $\mathrm{Sp}(n)\mathrm{Sp}(1)$. In this case, the fundamental forms $(\Phi_i)_{i=1,2,3}$ are parallel, but do not project. Nonetheless, the combined $3$-form $\Omega=\sum\Phi_i\wedge\Phi_i$ leads to the observed holonomy on the base.
\end{Remark}

\section{Submersions in SKT geometry}\label{Chp3}

For the remainder of this paper, we will work in a specific geometry with closed torsion, namely SKT, sometimes called pluriclosed, manifolds. Consider a Hermitian manifold $(M,g,J)$ and denote its fundamental form $\omega=g(\cdot ,J\cdot)$. Then there exists a unique Hermitian connection with skew-symmetric torsion
$T=\mathrm{d}^c\omega$,
where $\mathrm{d}^c\omega(X,Y,Z)=-\mathrm{d}\omega(JX,JY,JZ)$. This connection is called the Bismut, or sometimes Strominger, connection and agrees with the Levi-Civita connection if and only if $(M,g,J)$ is K\"ahler.

\begin{Definition}
A Hermitian manifold $(M,g,J)$ is called SKT if ${\rm d}T=0$.
\end{Definition}

The fully expanded term ``strong K\"ahler with torsion'' falsely suggests that SKT manifolds are K\"ahler. Instead, SKT manifolds are often called pluriclosed, referring to the fact that a~Hermitian manifold is SKT if and only if $\partial\bar{\partial}\omega=0$. We use the term SKT to emphasize the closed torsion involved.

We now want to understand Theorem~\ref{submthm} for SKT geometry.

\begin{Lemma}\label{SKTProjectable}
Let $(M,g,J)$ be an SKT manifold satisfying the conditions in Theorem~{\rm \ref{submthm}} and suppose that $\mathcal{H}$ is $J$-invariant. Let $\check{T}$ denote the torsion on the base $N$ of the submersion $\pi\colon M\to N$. Then $J$ projects to an almost complex structure $\check{J}$ on $N$ if and only if $(X\intprod T)|_{\Lambda^2\mathcal{H}}\in \Omega^{1,1}(M)$ for all $X\in \mathcal{V}$. In this case $\check{J}$ is integrable and parallel under $\nabla^{\check{T}}$. In particular, $\nabla^{\check{T}}$ is the Bismut connection on~$N$.
\end{Lemma}

\begin{proof}
We compute
\begin{align*}
g((\mathcal{L}_XJ)Y,Z)&=g([X,JY]-J[X,Y],Z)\\
&=g(\nabla_X(JY)-\nabla_{JY}X,Z)-T(X,JY,Z)\\
&\quad-g(J\nabla_XY-J\nabla_YX,Z)-T(X,Y,JZ)\\
&=g(\nabla_{JY}X-\nabla_Y(JX),Z)-T(X,JY,Z)-T(X,Y,JZ),
\end{align*}
where we have used that $\nabla J=0$. If $X\in \mathcal{V}$ and $Z\in\mathcal{H}$, then as the holonomy splits the covariant derivative summand vanishes. It follows that $J$ projects to an almost complex structure $\check{J}\in\operatorname{End} T\!N$ if and only if $T(X,JY,Z)+T(X,Y,JZ)=0$ for all $X\in\mathcal{V}$ and $Y,Z\in\mathcal{H}$ or equivalently $X\intprod T|_{\Lambda^2\mathcal{H}}\in\Omega^{1,1}(M)$.

Now that $J$ projects to an almost complex structure $\check{J}$ on $N$ by Theorem~\ref{holthm} we find $\mathrm{Hol}_0\bigl(\check{\nabla}\bigr)\subset \mathrm{Stab}\bigl(\check{J}\bigr)$. This shows $\check{\nabla}\check{J}=0$. It remains to show that $J$ is integrable. Observe that for any Hermitian connection $\nabla$ with torsion $T$, we have
\begin{align*}
N_J(X,Y,Z)&=g([X,Y]+J[JX,Y]+J[X,JY]-[JX,JY],Z)\\
&=g(\nabla_XY-\nabla_YX,Z)-T(X,Y,Z)\\
&\quad-g(\nabla_{JX}Y-\nabla_Y(JX),JZ)+T(JX,Y,JZ)\\
&\quad-g(\nabla_X(JY)-\nabla_{JY}X,JZ)+T(X,JY,JZ)\\
&\quad-g(\nabla_{JX}(JY)-\nabla_{JY}(JX),Z)+T(JX,JY,Z)\\
&=-T(X,Y,Z)+T(X,JY,JZ)+T(JX,Y,JZ)+T(JX,JY,Z).
\end{align*}
As this vanishes for the Bismut connection on $M$, the same holds for the connection $\check{\nabla}$ on $N$ with torsion given by $\pi^*\check{T}=T^\mathcal{H}$, implying that $J$ is integrable.
\end{proof}

In the rest of this section, we shall frame two submersion results on Bismut--Hermite--Einstein and sHKT manifolds as instances of Theorem~\ref{submthm}. We recall the definitions.

\begin{Definition}
An \emph{SKT} manifold $\bigl(M^{2n},g,J\bigr)$ is called BHE (Bismut--Hermite--Einstein) if
\begin{align*}
\rho^{\rm B}(X,Y)=\frac 12\sum_{i=1}^{2n}R^{\rm B}(X,Y,Je_i,e_i)=0,
\end{align*}
where $\{e_i\}$ is an orthonormal frame and $R^{\rm B}$ is the Bismut curvature tensor.
A \emph{hyperhermitian} manifold $\bigl(M^{4m},g,I_1,I_2,I_3\bigr)$ consists of three complex structures $(I_\alpha)_{\alpha=1,2,3}$ that satisfy the quaternionic relations and are all compatible with a single Riemannian metric~$g$. A hyperhermitian manifold is called \emph{HKT} if the Bismut connections of all three Hermitian structures agree, i.e., $T=T^{I_1}=T^{I_2}=T^{I_3}$. Finally, it is called \emph{sHKT} if, in addition, $\mathrm{d}T=0$.
\end{Definition}

As the Bismut connection is Hermitian, its holonomy is always contained in $\mathrm{U}(n)$. The condition $\rho^{\rm B}=0$ is equivalent to $\mathrm{Hol}\bigl(\nabla^{\rm B}\bigr)\subset \mathrm{SU}(n)$. In the HKT case, the Bismut connection preserves $I_1$, $I_2$, $I_3$ and~$g$ so its holonomy lies in $\mathrm{Sp}(m)\subset \mathrm{SU}(2m)$. In particular, any sHKT manifold is also BHE.

In \cite{Rigidity}, the authors prove that any non-K\"ahler Bismut--Hermitian--Einstein manifold admits a~non-vanishing vector field $V$ such that $V$ and $JV$ are parallel with respect to the Bismut connection. Furthermore, $J$~is preserved under the flow of $V$ and $JV$:
\begin{align}\label{Liederiv}
\mathcal{L}_VJ=\mathcal{L}_{JV}J=0.
\end{align}
They continue to prove that the distribution $\langle V,JV\rangle$ defines a Riemannian foliation. We rephrase their result using the submersion theorem.

\begin{Theorem}\label{thmBHE}
Let $(M,g,J)$ be a non-K\"ahler BHE manifold. Then there exists a locally defined Riemannian submersion $\pi\colon M\to N$ along the distribution spanned by $\{V, JV\}$. $J$ projects to an integrable complex structure $\check J$ on $N$ such that the Bismut connection $\check\nabla$ on $\bigl(N, g_N,\check{J}\bigr)$ and~$\nabla$ on $(M,g,J)$ are related via
\begin{align*}
\check\nabla_XY=\pi_*\bigl(\nabla_{\overline{X}}\overline{Y}\bigr).
\end{align*}
\end{Theorem}

\begin{proof}
As the two vector fields $V$ and $JV$ are parallel under the Bismut connection, we may consider the decomposition $TM=\mathcal{V}\oplus\mathcal{H}$ where $\mathcal{V}=\langle V, JV\rangle$. Furthermore,
\begin{align*}
\mathrm{d}V^\flat=\sum_{i=1}^ne_i\wedge\nabla^g_{e_i}V^\flat=\sum_{i=1}^ne_i\wedge\big(\nabla_{e_i}V^\flat-T(e_i,V,\cdot)\big)=V\intprod T.
\end{align*}
Hence, by \cite[Proposition~2.3]{FinosHKT}, $JV\intprod V\intprod T=JV\intprod \mathrm{d}V^\flat=0$. This implies that $T\in \Lambda^2\mathcal{H}\wedge V\oplus \Lambda^3\mathcal{H}$ and we can employ Theorem~\ref{submthm} to obtain a locally defined Riemannian submersion $\pi\colon M\to N$ along the distribution $\langle V,JV\rangle$.

Then \eqref{Liederiv} shows that the complex structure $J$ is constant along the fibers and hence projectable. As seen in the proof of Lemma~\ref{SKTProjectable}, this implies that we may furnish the base $N$ with an integrable complex structure whose Bismut connection $\check\nabla$ is the projection of the Bismut connection of the total space.
\end{proof}

In the case of sHKT, the parallel vector fields $V$ associated to $I_\alpha$, $\alpha=1,2,3$, coincide~\cite{FinosHKT}. In other words, there exists a vector field $V$ such that $V$, $I_1V$, $I_2V$, $I_3V$ are parallel vector fields for $\nabla^{\rm B}$. They show the following statement in dimension $4m=8$:

\begin{Theorem}[\cite{FinosHKT}]
Let $\bigl(M^8,g,I_1,I_2,I_3\bigr)$ be an $8$-dimensional sHKT manifold that is not hyperk\"ahler. Then the distribution spanned by $V$, $I_1V$, $I_2V$, $I_3V$ is integrable and satisfies
\begin{align}\label{conical}
\mathcal{L}_V\omega_i=0,\qquad \mathcal{L}_{I_i V}\omega_i=0,\qquad \mathcal{L}_{I_i V}\omega_j=-\omega_k,
\end{align}
where $(ijk)$ is an even permutation of $(123)$.
\end{Theorem}

In particular, this shows that $\bigl(M^8,g,I_1,I_2,I_3\bigr)$ is conical in the sense of \cite{CortesHasegawa} with Euler vector field $2 V$. Indeed, \eqref{conical} shows by the argument in \cite[Lemma~2.2]{PPSHypercomplex} that $\nabla^{\rm O}V=\frac 12 \operatorname{Id}$, where the Obata connection $\nabla^{\rm O}$ is the unique torsion-free connection preserving the hyperhermitian structure $(g,I_1,I_2,I_3)$, see~\cite{Obata}. Note that we may argue with the fundamental forms as $V$, $I_1$, $I_2$, $I_3$ are parallel with respect to the Bismut connection and thereby Killing. Conversely, given a conical sHKT structure with Euler vector field~$2V$ the authors in~\cite{CortesHasegawa} show that~\eqref{conical} is satisfied. Moreover, it is shown that the distribution spanned by $V$, $I_1V$, $I_2V$, $I_3V$ is integrable.

\begin{Lemma}\label{anticyclic}
The conditions in~\eqref{conical} for all cyclic permutations $(ijk)$ of $(123)$ imply ${\mathcal{L}_{I_j V}\omega_i\!=\!\omega_k}$.
\end{Lemma}

\begin{proof}
As $\mathcal{L}_{I_k V}\omega_k=0$, we can commute $\mathcal{L}_{I_k V}$ with $I_k$ and, hence,
\begin{align*}
(\mathcal{L}_{I_k V}\omega_j)(X,Y)&=(I_k V)(\omega_j(X,Y))-\omega_j(\mathcal{L}_{I_k V}X,Y)-\omega_j(X,\mathcal{L}_{I_k V}Y)\\
&=-(I_k V)(\omega_i(X,I_kY))+\omega_i(\mathcal{L}_{I_k V}X,I_kY)+\omega_i(X,I_k\mathcal{L}_{I_k V}Y)\\
&=-(I_k V)(\omega_i(X,I_kY))+\omega_i(\mathcal{L}_{I_k V}X,I_kY)+\omega_i(X,\mathcal{L}_{I_k V}I_kY)\\
&=-(\mathcal{L}_{I_k V}\omega_i)(X,I_kY)=\omega_j(X,I_kY)=\omega_i(X,Y).\tag*{\qed}
\end{align*}\renewcommand{\qed}{}
\end{proof}

With that, we reach the same conclusion as \cite{FinosHKT,CortesHasegawa} in terms of Theorem~\ref{submthm}.

\begin{Theorem}\label{thmsHKT}
Let $(M,g,I_1,I_2,I_3)$ be a non-hyperk\"ahler conical sHKT manifold. Then there exists a locally defined Riemannian submersion $\pi\colon M\to N$ with totally geodesic leaves spanned by $\{V,IV,JV,KV\}$. The base $N$ admits a quaternionic structure $\check{\Omega}$ and a quaternionic connection $\check\nabla$ with skew-symmetric torsion.
\end{Theorem}

\begin{proof}
As seen above, the $4$-dimensional distribution $\mathcal{V}=\langle V,IV,JV,KV\rangle$ is a trivial representation of the holonomy. By Remark~\ref{RemIntegrable}, vanishing of the $\Lambda^2\mathcal{V}\otimes\mathcal{H}$-part of the torsion in Theorem~\ref{submthm} is equivalent to $\mathcal{V}$ being integrable, which we have seen for conical sHKT structures. We obtain a locally defined Riemannian submersion $\pi\colon M\to N$. Now consider the quaternionic $4$-form $\Omega=\omega_1\wedge\omega_1+\omega_2\wedge\omega_2+\omega_3\wedge\omega_3\in \Lambda^4T^*M$. As the $\omega_i$ are parallel under the Bismut connection, so is $\Omega$. Now consider $\mathcal{L}_W\Omega$ for $W=\sum_{k=0}^3f_kI_kV\in\mathcal{V}$, where we set $I_0=\mathrm{id}$ and $f_i\in C^\infty(M)$. We have for $X,Y\in\mathcal{H}$ and $(ijk)$ an even permutation of $(123)$
\begin{align*}
(\mathcal{L}_W\omega_j)(X,Y)&=W(\omega_j(X,Y))-\omega_j(\mathcal{L}_WX,Y)-\omega_j(X,\mathcal{L}_WY)\\
&=\sum_{\mu=0}^3\big((f_\mu \mathcal{L}_{I_\mu V}\omega_j)(X,Y)+\omega_j(\mathrm{d}f_\mu (X)I_\mu V,Y)
 +\omega_j(X,\mathrm{d}f_\mu(Y)I_\mu V)\big)\\
&=\sum_{\mu=1}^3f_\mu(\mathcal{L}_{I_\mu V}\omega_j)(X,Y)=(f_i\omega_k-f_k\omega_i)(X,Y),
\end{align*}
where we have used that $\omega_j(\mathcal{V},\mathcal{H})=0$, the identities \eqref{conical}, and Lemma~\ref{anticyclic}. Thus, restricted to~$\Lambda^4\mathcal{H}$
\begin{align*}
\mathcal{L}_W\Omega&=\sum_{j=1}^3\mathcal{L}_W(\omega_j\wedge\omega_j) =2\sum_{j=1}^3(\mathcal{L}_W\omega_j)\wedge\omega_j=2
\!\cyclic{ijk} \!(f_i\omega_k-f_k\omega_i)\wedge\omega_j\\
&=2\!\cyclic{ijk}\!f_i(\omega_k\wedge\omega_j-\omega_j\wedge\omega_k)=0.
\end{align*}
As the $\omega_i$ are non-vanishing only on $\Lambda^2\mathcal{V}\oplus\Lambda^2\mathcal{H}$ $\Omega$ projects to a $4$-form $\check{\Omega}\in\Lambda^4T^*N$. As the stabilizer of such a $4$-form is $\mathrm{Stab}\bigl(\check\Omega\bigr)=\mathrm{Sp}(m-1)\mathrm{Sp}(1)$, we have that the holonomy $\mathrm{Hol}_0\bigl(\check{\nabla}\bigr)\subset \mathrm{Sp}(m-1)\mathrm{Sp}(1)$ by Theorem~\ref{holthm}.
\end{proof}

\begin{Remark}
In \cite{KSstring}, the authors investigate several further geometries akin to Theorems~\ref{thmsHKT} and~\ref{thmBHE}. In their setup, they independently arrive at similar conclusions as Theorem~\ref{submthm}, see \cite[Lemma~2.2]{KSstring}.
\end{Remark}

\begin{Remark}
Theorems~\ref{thmBHE} and~\ref{thmsHKT} are particularly useful given regularity assumption of the local submersions. For instance, see the study of low-dimensional BHE and sHKT manifolds in~\cite{ALL} and~\cite{Papadopoulos}, respectively.
\end{Remark}

\section{Submersions on SKT Samelson spaces}

In this section, we consider SKT structures on semi-simple Lie groups. These were first found on $\mathrm{SO}(9)$ in \cite{FGSKT} and by \cite{Pham} on $G_2$. We follow the general construction and notation in~\cite{LauMon}.

Let $G$ be a simply connected, compact, semi-simple Lie group of even rank. Set the Lie algebra of $G$ as $\mathfrak{g}=\mathfrak{t}\oplus\mathfrak{q}$ with $\mathfrak{t}$ its Cartan subalgebra. Further consider the root space decomposition
\begin{align*}
\mathfrak{g}^{\mathbb{C}}=\mathfrak{t}^{\mathbb{C}}\oplus\sum_{\alpha\in\Delta^+}(\mathfrak{g}_\alpha\oplus\mathfrak{g}_{-\alpha}),
\end{align*}
where $\Delta^+$ is a choice of positive roots. For any $\alpha\in\Delta^+$, we may choose $E_{\pm\alpha}\in \mathfrak{g}_{\pm\alpha}$ and $H_\alpha\in \mathfrak{t}^{\mathbb{C}}$ such that
\begin{gather*}
[E_\alpha,E_{-\alpha}]=H_\alpha,\qquad B(E_\alpha,E_{-\alpha})=1,\qquad
[H,E_\alpha]=\alpha(H)E_\alpha, \qquad B(H,H_\alpha)=\alpha(H)
\end{gather*}
for all $H\in\mathfrak{t}^{\mathbb{C}}$ and where $B(X,Y)=\mathrm{tr}(\mathrm{ad}(X)\circ\mathrm{ad}(Y))$ is the Killing form.

We consider a~left-invariant complex structure $J\colon \mathfrak{g}\to\mathfrak{g}$ due to Samelson given as $J=J|_\mathfrak{t}+J|_\mathfrak{q}$, where $J|_\mathfrak{t}$ is given by any complex structure on $\mathfrak{t}\cong \mathbb R^{\operatorname{rk}G}$ and $J|_\mathfrak{q}$ is defined as
\begin{align*}
JE_{\pm\alpha}=\pm iE_{\pm\alpha}
\end{align*}
for $\alpha \in \Delta^+$. Note that the defined $J|_\mathfrak{q}$ restricts to a complex structure on the compact real root spaces $\mathfrak{g}^\mathbb{R}_\alpha=(\mathfrak{g}_\alpha\oplus\mathfrak{g}_{-\alpha})\cap\mathfrak{g}$. Up to biholomorphism, this does not depend on the choice of maximal torus and positive roots $\Delta^+$. In fact, the biholomorphism class of $(G,J)$ is uniquely determined by the choice of $J|_{\mathfrak{t}}\colon \mathfrak{t}\to\mathfrak{t}$ up to the action of a finite subgroup of $\mathrm{Aut}(G)$ \cite{Pittie}.

Let $\mathfrak{g}=\mathfrak{g}_1\oplus\cdots\oplus\mathfrak{g}_s$ be the decomposition of $\mathfrak{g}$ into its simple factors and $\mathfrak{t}_i$, $\Delta_i^+$ their corresponding maximal tori and positive roots.
We consider the left $G$- and right $T$-invariant metrics on $\mathfrak{g}$ given by
\begin{align*}
g=-\sum_{i=1}^s\lambda_i\Biggl(B|_{\mathfrak{t}_i\times\mathfrak{t}_i}+\sum_{\alpha\in \Delta_i^+}c_\alpha B|_{\mathfrak{g}_\alpha^\mathbb{R}\times\mathfrak{g}_\alpha^\mathbb{R}}\Biggr),
\end{align*}
for constants $\lambda_i,c_\alpha>0$. For each $J$-indecomposable component of $G$ and thereby $\mathfrak{t}$, there is up to uniform scaling a unique set of $\lambda_i$ such that $g$ is compatible with $J$. In \cite{LauMon}, the authors show that all pluriclosed Hermitian metrics that are left $G$- and $\mathrm{Ad}(T)$-invariant are obtained this way. We will assume $g$ to be $J$-compatible in the following.

On reductive homogeneous manifolds, $G/H$ invariant connections are in bijective correspondence with $\mathrm{Ad}(H)$-invariant linear maps $\Lambda\colon \mathfrak{m}\times\mathfrak{m}\to \mathfrak{m}$, compare \cite[Chapter~10]{KobNom2}.
In our case, $H=\{e\}$, $\mathfrak{m}=\mathfrak{g}$, and any metric connection with skew-torsion $T$ can be characterized via the corresponding $(0,3)$-tensor $g(\Lambda_XY,Z)\in \mathfrak{g}^*\otimes\Lambda^2\mathfrak{g}^*\subset\otimes^3\mathfrak{g}^*$,
\begin{align*}
g(\Lambda_XY,Z)&=g(\Lambda^{g}_XY,Z)+\tfrac 12 T(X,Y,Z)\\
&=\tfrac 12  (g([Z,X],Y)+g([Z,Y],X)+g([X,Y],Z) )+\tfrac 12 T(X,Y,Z).
\end{align*}

Further, the correspondence identifies the holonomy algebra as the subalgebra of $\End(\mathfrak{g})$ generated by
\begin{align}\label{holonomy_generators}
\Lambda_{X_1}\Lambda_{X_2}\cdots\Lambda_{X_k}R(X,Y)
\end{align}
for any $X_1,\dots,X_k,X,Y\in\mathfrak{g}$ and where $R(X,Y)=[\Lambda_X,\Lambda_Y]-\Lambda_{[X,Y]}$ is the curvature operator.

\begin{Lemma}\label{TorusSplit}
The maximal torus $\mathfrak{t}$ is a trivial representation of the Bismut holonomy,
\begin{align*}
\mathfrak{g}=\mathfrak{t}\oplus \Biggl(\sum_{\alpha\in\Delta^+}\mathfrak{g}^\mathbb{R}_\alpha\Biggr).
\end{align*}
Furthermore, the decomposition $\mathfrak{g}=\mathfrak{g}_1\oplus\cdots\oplus\mathfrak{g}_s$ of $\mathfrak{g}$ into simple Lie algebras is preserved under the holonomy.
\end{Lemma}

\begin{proof}
The Bismut connection has torsion
\begin{align}
T(X,Y,Z)&=\mathrm{d}^c\omega(X,Y,Z)=-\mathrm{d}\omega(JX,JY,JZ)\nonumber\\
&=\omega([JX,JY],JZ)+\omega([JY,JZ],JX)+\omega([JZ,JX],JY)\nonumber\\
&=-g([JX,JY],Z)-g([JY,JZ],X)-g([JZ,JX],Y),\label{LieBismutTorsion}
 \end{align}
and, thus, corresponds to the Nomizu operator
\begin{align}
g(\Lambda_XY,Z)=\tfrac 12 (&g([X,Y],Z)+g([Z,X],Y)-g([Y,Z],X)\nonumber\\
&-g([JX,JY],Z)-g([JZ,JX],Y)-g([JY,JZ],X)). \label{Lambda}\
\end{align}
Observe that by \eqref{holonomy_generators} if $\Lambda_X$ preserves a given subspace for all $X\in\mathfrak{g}$ so does the holonomy. We show the splitting on the complexified $\mathfrak{g}^\mathbb{C}$ and $\Lambda$. If all three vectors $X,Y,Z\in\mathfrak{t}^\mathbb{C}$, then so are $JX$, $JY$, $JZ$. Hence, all commutators in \eqref{Lambda} vanish and so does~$\Lambda$. Now suppose any two vectors of $X$, $Y$, $Z$ are in $\mathfrak{t}^\mathbb{C}$. Then commutators of these two vanish directly. Commutators of these with a vector in some root space $\mathfrak{g}_\alpha$ map into $\mathfrak{g}_\alpha$ and, thus, pair trivially with any vector not in $\mathfrak{g}_{-\alpha}$. It remains to investigate the case where only $Y\in \mathfrak{t}^\mathbb{C}$. By previous consideration, the only possible non-trivial components of $\Lambda$ can be obtained if $X\in \mathfrak{g}_\alpha$ is in some root space and $Z\in \mathfrak{g}_{-\alpha}$. We separate two cases: For $\alpha\in \Delta^+$,
\begin{align*}
g(\Lambda_XY,Z)&=\tfrac 12 (g([X,Y],Z)+g([Z,X],Y)-g([Y,Z],X)\\
&\quad-g([iX,JY],Z)-g([-iZ,iX],Y)-g([JY,-iZ],X))\\
&=\tfrac 12 (-\alpha(Y)g(X,Z)+g([Z,X],Y)+\alpha(Y)g(Z,X)\\
&\quad+i\alpha(JY)g(X,Z)-g([Z,X],Y)-i\alpha(JY)g(Z,X))=0,
\end{align*}
and for $-\alpha\in \Delta^+$,
\begin{align*}
g(\Lambda_XY,Z)&=\tfrac 12 (g([X,Y],Z)+g([Z,X],Y)-g([Y,Z],X)\\
&\quad-g([-iX,JY],Z)-g([iZ,-iX],Y)-g([JY,iZ],X))\\
&=\tfrac 12 (\alpha(Y)g(X,Z)+g([Z,X],Y)-\alpha(Y)g(Z,X)\\
&\quad+i\alpha(JY)g(X,Z)-g([Z,X],Y)-i\alpha(JY)g(Z,X)) =0.
\end{align*}
To prove the second assertion, we consider $Y\in\mathfrak{g}_i$. If $Y\in\mathfrak{t}$, then we have just shown that $Y$ is parallel. Hence, let $Y\in\mathfrak{g}_\alpha$ for some root in $\Delta_i$. By skew-symmetry, we may assume $Z\in \mathfrak{g}_\beta$ for some root $\beta\in \Delta$. Note that $J$ preserves $\mathfrak{g}_\alpha$ and $\mathfrak{g}_\beta$. If $\alpha+\beta=0$, $X\in \mathfrak{g}_{-\alpha}\subset \mathfrak{g}_i$ as well. If $\alpha+\beta\neq 0$, the complex structure $J$ preserves $\mathfrak{g}_{-\alpha-\beta}$ and all commutators in \eqref{Lambda} vanish unless $X\in \mathfrak{g}_{-\alpha-\beta}$. That implies $\alpha+\beta\in \Delta$ and, thus, $\beta\in \Delta_i$ is in the same simple component.
\end{proof}

We now want to restrict to homogeneous SKT structures. They were described in \cite{LauMon} in the following theorem.

\begin{Theorem}[\cite{LauMon}]\label{Thm:Pluriclosed}
The Hermitian manifold $(G,J,g)$ is pluriclosed if and only if $c_{\alpha+\beta}=c_\alpha+c_\beta-1$ for all $\alpha,\beta\in \Delta^+$ such that $\alpha+\beta\in\Delta$.

Equivalently, for any positive root $\alpha=\sum_{j=1}^{\operatorname{rk}G}k_j\alpha_j$, we have
\begin{align*}
c_\alpha=1+\sum_{j=1}^{\operatorname{rk}G}k_j(c_{\alpha_j}-1).
\end{align*}
where $c_{\alpha_j}\in \R$, for $\alpha_j$ a simple root, are free parameters so long as all $c_\alpha>0$.
\end{Theorem}

Consider the set $I_{\max}$ of simple roots $\alpha_i$ such that $c_{\alpha_i}=1$. Then by construction for all roots~$\alpha$ in the span of $I$
\begin{align*}
\Delta_{I_{\max}}\coloneqq \Biggl\{\sum_{\alpha_i\in I_{\max}} k_i\alpha_i\mid k_i\in\mathbb{N}\Biggr\}\cap \Delta,
\end{align*}
we also have $c_{\alpha}=1$.

\begin{Theorem}\label{SamelsonHolonomy}
Let $(G,J,g)$ be a pluriclosed Hermitian structure as above. Then $\mathfrak{g}$ splits as a~representation of the Bismut holonomy as
\begin{align*}
\mathfrak{g}=\mathfrak{t}\oplus\bigoplus_{\alpha\in \Delta_{I_{\max}}}\mathfrak{g}_{\alpha}^\mathbb{R}\oplus\bigoplus_{i=1}^s \Biggl(\sum_{\alpha\in\Delta_i^+\setminus\Delta_{I_{\max}}}\mathfrak{g}_\alpha^\mathbb{R}\Biggr).
\end{align*}
Here the summand $\mathfrak{t}\oplus\bigoplus_{\alpha\in\Delta_{I_{\max}}}\mathfrak{g}_{\alpha}^{\mathbb{R}}$ is trivial.
\end{Theorem}

\begin{proof}
 Fix $\alpha\in\Delta_{I_{\max}}\cap\Delta_i^+$. Let $Y\in \mathfrak{g}_{\alpha}$ and $Z\in \mathfrak{g}_\beta$, where $\alpha\pm\beta\neq0$. We consider $g(\Lambda_XY,Z)$ as in \eqref{Lambda}. Observe that $X\in \mathfrak{g}_\gamma$ with $\alpha+\beta+\gamma=0$ or otherwise the commutators map to the wrong root spaces and all summands in $\Lambda$ vanish. By Lemma~\ref{TorusSplit}, we may also disregard $Z\in\mathfrak{t}^\mathbb{C}$ as any metric connection admits skew-symmetry in $Y$ and $Z$. Suppose $\beta\in \Delta^+$. Then $-\gamma=\alpha+\beta\in \Delta^+$ and we find
\begin{align*}
g(\Lambda_XY,Z)&=\tfrac 12 (g([X,Y],Z)+g([Z,X],Y)-g([Y,Z],X)\\
&\quad-g([-iX,iY],Z)-g([iZ,-iX],Y)-g([iY,iZ],X))=0
\end{align*}
Now assume $-\beta\in \Delta^+$. If $\gamma=-\alpha-\beta\in\Delta^+$, we have
\begin{align}
g(\Lambda_XY,Z)&=\tfrac 12 (g([X,Y],Z)+g([Z,X],Y)-g([Y,Z],X)\notag\\
&\quad-g([iX,iY],Z)-g([-iZ,iX],Y)-g([iY,-iZ],X))\notag\\
&=-c_{-\beta}\lambda_i B([X,Y],Z)+c_{\gamma}\lambda_iB([Y,Z],X)\notag\\
&=-\lambda_i(c_{-\beta}-c_{\gamma})B([X,Y],Z).\label{ppm}
\end{align}
Hence, Theorem~\ref{Thm:Pluriclosed} implies
\begin{align*}
c_{-\beta}=c_{\alpha-\alpha-\beta}=c_{\alpha+\gamma}=c_{\alpha}+c_{\gamma}-1
\end{align*}
and $g(\Lambda_XY,Z)$ vanishes if and only if $c_\alpha=1$.
If $-\gamma=\alpha+\beta\in \Delta^+$, we have
\begin{align}
g(\Lambda_XY,Z)=-\lambda_i(c_\alpha-c_{-\gamma})B([X,Y],Z)\label{mpm}
\end{align}
instead and it vanishes if and only if $c_{-\beta}=1$. However, $-\gamma\in\Delta^+$ implies that $\alpha=-\beta-\gamma$ is the sum of positive roots and, hence, $-\beta,-\gamma\in \Delta_{I_{\max}}$. Therefore, $c_{-\beta}=1$, as desired.

As the complex structure $J$ appears only pairwise in \eqref{Lambda}, we obtain the same results when inverting the signs of $\alpha$, $\beta$ and $\gamma$. Hence, we have shown that the root pairs $\mathfrak{g}_\alpha\oplus\mathfrak{g}_{-\alpha}$ for $\alpha\in\Delta_{I_{\max}}$ split off in the complexified holonomy representation and hence the corresponding $2$-dimensional real subspaces $\mathfrak{g}_{\alpha}^{\mathbb{R}}=(\mathfrak{g}_{\alpha}\oplus\mathfrak{g}_{-\alpha})\cap \mathfrak{g}$.

It remains to show that the maximal torus $\mathfrak{t}$ acts trivially on $\mathfrak{g}_\alpha^{\mathbb{R}}$. For $H\in \mathfrak{t}^{\mathbb{C}}$, $Y\in \mathfrak{g}_{\alpha}$ and $Z\in\mathfrak{g}_{-\alpha}$, we have
\begin{align}
g(\Lambda_HY,Z)=-\lambda_i(c_\alpha-1)B([H,Y],Z).\label{tpm}
\end{align}
This vanishes for $\alpha\in \Delta_{I_{\max}}$. Now by skew-symmetry of $\Lambda_X$ these are all possibilities and $\Lambda_XY=0$ for all $X\in\mathfrak{g}$ and $Y\in\mathfrak{g}_\alpha\oplus\mathfrak{g}_{-\alpha}$.
\end{proof}

In general, we cannot guarantee that $c_\alpha\neq 1$ for $\alpha\in\Delta^+\setminus\Delta_{I_{\max}}$ and therefore the non-vanishing of \eqref{ppm}, \eqref{mpm} and \eqref{tpm}. However, if we assume that $c_{\alpha_i}<1$ for all $\alpha_i\not\in I_{\max}$, then the same holds for all $\alpha\in\Delta^+\setminus\Delta_{I_{\max}}$.

\begin{Corollary}
If $0<c_\alpha< 1$ for all $\alpha\in \Delta^+\setminus\Delta_{I_{\max}}$, then the real holonomy representations
\begin{align*}
V_i=\sum_{\alpha\in\Delta_i^+\setminus\Delta_{I_{\max}}}\mathfrak{g}_\alpha^\mathbb{R}
\end{align*}
are irreducible.
\end{Corollary}

\begin{proof}
We may assume that $\mathfrak{g}$ is simple and $\lambda_1=1$. Consider the two complex $\mathfrak{hol}(\nabla)^{\mathbb{C}}$-representations $W=\sum_{\alpha\in\Delta^+\setminus\Delta_{I_{\max}}}\mathfrak{g}_\alpha$ and $\overline{W}=\sum_{\alpha\in\Delta^+\setminus\Delta_{I_{\max}}}\mathfrak{g}_{-\alpha}$. We show that they are irreducible. Let $X\in \mathfrak{g}_\alpha$ and $Y\in \mathfrak{g}_{\beta}$, $\alpha,\beta\in\Delta^+\setminus\Delta_{I_{\max}}$,
\begin{align*}
R(Y,\overline{X},X,\overline{Y})&=g\bigl(\Lambda_{Y}\Lambda_{\overline{X}}X,\overline{Y}\bigr) -g\bigl(\Lambda_{\overline{X}}\Lambda_{Y}X,\overline{Y}\bigr)-g\bigl(\Lambda_{[Y,\overline{X}]}X,\overline{Y}\bigr)\\
&=g\bigl(\Lambda_{Y}H,\overline{Y}\bigr)g\bigl(\Lambda_{\overline{X}}X,H\bigr)-g\bigl(\Lambda_{\overline{X}}E_{\alpha+\beta}, \overline{Y}\bigr)g(\Lambda_{Y}X,E_{-(\alpha+\beta)})\\
&\quad
  -g\bigl(\Lambda_{[Y,\overline{X}]}X,\overline{Y}\bigr),
\end{align*}
where $H\in \mathfrak{t}^{\mathbb{C}}$ is chosen appropriately and $g(E_{\alpha+\beta},E_{-(\alpha+\beta)})=1$ for $E_{\pm(\alpha+\beta)}\in\mathfrak{g}_{\pm(\alpha+\beta)}$. Now the first term vanishes as $\Lambda_XH=0$ for any $H\in\mathfrak{t}^{\mathbb{C}}$. Using \eqref{ppm} and \eqref{mpm}, we find
\begin{align*}
g\bigl(\Lambda_{\overline{X}}E_{\alpha+\beta},\overline{Y}\bigr) g(\Lambda_{Y}X,E_{-(\alpha+\beta)})&=(c_\beta-1)(c_\alpha-1)B\bigl([\overline{X},E_{\alpha+\beta}],\overline{Y}\bigr) B([Y,X],E_{-(\alpha+\beta)})\\
&=-\frac{(c_\beta-1)(c_\alpha-1)}{c_{\alpha+\beta}}B\bigl([\overline{X},\overline{Y}],[X,Y]\bigr)
\end{align*}
and
\begin{align*}
g\bigl(\Lambda_{[Y,\overline{X}]}X,\overline{Y}\bigr)=-B\bigl([Y,\overline{X}],[X,\overline{Y}]\bigr)\begin{cases}(c_\alpha-1), & \beta-\alpha>0,\\(c_\beta-1), &\beta-\alpha<0.\end{cases}
\end{align*}
As both $(c_\alpha-1),(c_{\beta}-1)>0$, we have
\begin{align*}
R\bigl(Y,\overline{X},X,\overline{Y}\bigr) =\frac{(c_\beta-1)(c_\alpha-1)}{c_{\alpha-\beta}}B\bigl(\overline{[X,Y]},[X,Y]\bigr) -\left\{\begin{aligned}(c_\alpha-1)\\(c_\beta-1)\end{aligned}\right\} B\bigl(\overline{[X,\overline{Y}]},[X,\overline{Y}]\bigr)>0
\end{align*}
as long as $\alpha+\beta$ or $\alpha-\beta$ is a root.
Therefore, $R\bigl(Y,\overline{X}\bigr)X\in\mathfrak{g}_{\beta}\setminus\{0\}$. As $\alpha$ and $\beta$ are arbitrary roots in $\Delta^+\setminus\Delta_{I_{\max}}$, the representation $W$ is irreducible. The same holds for $\overline{W}=\sum_{\alpha\in\Delta^+\setminus\Delta_{I_{\max}}}\mathfrak{g}_{-\alpha}$.

We claim that the two representations $W$ and $\overline{W}$ are non-isomorphic as $\mathfrak{hol}(\nabla)$-represen\-tations. Assume by contradiction that $T\colon W\to \overline{W}$ is an invariant isomorphism. Then for every $h\in \mathfrak{hol}(\nabla)$ we have $T\circ h|_W=h|_{\overline{W}}\circ T$. In particular, for every eigenvector $X\in W$ of $h|_{W}$, the image $TX\in \overline{W}$ is an eigenvector for $h|_{\overline{W}}$ with the same eigenvalue. Let $\alpha\in \Delta^+\setminus\Delta_{I_{\max}}$ be the highest root. Note that such a root exists if $\Delta_{I_{\max}}\subsetneq \Delta^+$. Let $X^\mathbb{R}=X-\overline{X}\in \mathfrak{g}_{\alpha}^{\mathbb{R}}\subset\mathfrak{g}_{\alpha}\oplus\mathfrak{g}_{-\alpha}$ with $B\bigl(X,\overline{X}\bigr)=1$ and consider $h=R\bigl(X^\mathbb{R},JX^\mathbb{R}\bigr)$. For any $Y\in\mathfrak{g}_\beta$, $\beta\in \Delta^+\setminus\Delta_{I_{\max}}$, $g\bigl(Y,\overline{Y}\bigr)=1$, we~find
\begin{align*}
R(X,JX)Y&=R\bigl(X,JX,Y,\overline{Y}\bigr)Y=2iR\bigl(X,\overline{X},Y,\overline{Y}\bigr)Y\\
&=2iY\bigl(g\bigl(\Lambda_XE_{-(\alpha-\beta)},\overline{Y}\bigr) g(\Lambda_{\overline{X}}Y,E_{\alpha-\beta})-g\bigl(\Lambda_{[X,\overline{X}]}Y,\overline{Y}\bigr)\bigr).
\end{align*}
Observe that if $\beta\neq \alpha$, then $\alpha-\beta\in \Delta^+$ and the first term vanishes. Hence,
\begin{align*}
R(X,JX)Y&=2i(c_\beta-1)B\bigl([X,\overline{X}],[Y,\overline{Y}]\bigr)Y=2i(c_\beta-1)\langle\alpha,\beta\rangle Y.
\end{align*}
As $\alpha$ is the highest root, $\langle\alpha,\beta\rangle\geq 0$ and strictly positive if~$\alpha$ and~$\beta$ are not orthogonal. In particular, the eigenvalues of $R\bigl(X^\mathbb{R},JX^\mathbb{R}\bigr)$ on $W$ are multiples of~$-i$. The same argument shows that $\mathfrak{g}_{-\beta}$ is a $R(X,JX)$-eigenspace to the eigenvalue $-2i(c_\beta-1)\langle\alpha,\beta\rangle\in i\mathbb{R}_+$. In conclusion, the eigenvalues of~$R(X,JX)$ on~$W$ and $\overline{W}$ differ contradicting a $\mathrm{hol}(\nabla)$-invariant isomorphism $T\colon W\to\overline{W}$.

The final step is to assume there is a non-trivial $\mathfrak{hol}(\nabla)$-subrepresentation $U\subset V_i$. Then $U^\mathbb{C}=U\otimes\mathbb{C}\subset V_i\otimes \mathbb{C}$ is a $\mathfrak{hol}(\nabla)^\mathbb{C}$-invariant subrepresentation. As $W$ and $\overline{W}$ are not isomorphic, either $U^\mathbb{C}=W$ or $U^\mathbb{C}=\overline{W}$. But, as $U$ is real, $U^\mathbb{C}$ is invariant under complex conjugation in contradiction to either $U^\mathbb{C}=W$ or $U^{\mathbb{C}}=\overline{W}$.
\end{proof}

\begin{Remark}
 If all simple roots satisfy $c_{\alpha_i}=1$, then the metric $g$ is biinvariant. It is a well known result that $(G,J,g)$ with $g$ biinvariant is Bismut-flat and, hence, the holonomy representation is trivial. In this case, the SKT structure is obviously also BHE. By \cite[Corollary~5.1\,(iii)]{FGSKT}, any other metric fails to be BHE showing that the associated submersions below are distinct from the ones in Theorem~\ref{thmBHE}.
\end{Remark}

\begin{Corollary}\label{SamelsonSubmersion}
Let $I\subset I_{\max}$ and $H_I$ the connected subgroup of $G$ with Lie algebra $\mathfrak{h}_I=\mathfrak{t}\oplus\bigoplus_{\alpha\in \Delta_I}\mathfrak{g}_{\alpha}^\mathbb{R}$. Then the map
$G\rightarrow G/H_I$
is a Hermitian submersion. The Bismut connection projects to the Bismut connection on $N=G/H_I$.
\end{Corollary}

\begin{proof}
Let $\mathfrak{h}=\mathfrak{h}_I^\mathbb{C}$ and set $\mathfrak{g}^\mathbb{C}= \mathfrak{h}\oplus\mathfrak{p}$. We prove that the Bismut torsion $T^{\rm B}(X,Y,Z)=0$ for $X,Y\in \mathcal{V}^\mathbb{C}=\mathfrak{h}$ and $Z\in\mathcal{H}^\mathbb{C}= \mathfrak{p}$. Let $X\in\mathfrak{g}_\alpha$, $Y\in \mathfrak{g}_\beta$ and $Z\in \mathfrak{g}_\gamma$, where we allow $\mathfrak{g}_0\coloneqq \mathfrak{t}^\mathbb{C}$. Observe that $J$ preserves these spaces. Thus, either all terms in \eqref{LieBismutTorsion} vanish individually and $T(X,Y,Z)=0$ or $\alpha+\beta+\gamma=0$. However, this implies that if $\alpha,\beta\in \pm\Delta_I\cup\{0\}$ so is $-\gamma$. Hence, we obtain a Riemannian submersion as in Theorem~\ref{submthm}.

To apply Lemma~\ref{SKTProjectable}, we further show that $T(X,JY,JZ)=T(X,Y,Z)$ for all $X\in\mathfrak{g}_\alpha \subset\mathfrak{h}$ and $Y\in\mathfrak{g}_\beta\subset\mathfrak{p}$ and $Z\in \mathfrak{g}_\gamma\subset\mathfrak{p}$. As before, we may assume $\alpha+\beta+\gamma=0$. Suppose $\beta,\gamma\in\Delta^+$. Then $-\alpha=\beta+\gamma$ is the sum of two positive roots, implying $\beta,\gamma\in \Delta_I$, in contradiction to $Y,Z\in \mathfrak{p}$. The same holds if $-\beta,-\gamma\in\Delta^+$. Hence, we may assume $\beta\in \Delta^+$ and $-\gamma\in \Delta^+$.
By~\eqref{LieBismutTorsion},
\begin{align*}
T(X,JY,JZ)&=-g([JX,Y],JZ)+g([Y,Z],X)-g([Z,JX],JY)\\
&=ig([JX,Y],Z)-g([iY,iZ],X)-g([Z,JX],Y)\\
&=g([JX,JY],Z)+g([JY,JZ],X)+g([JZ,JX],Y)\\
&=T(X,Y,Z).\tag*{\qed}
\end{align*}\renewcommand{\qed}{}
\end{proof}

\begin{Remark}
The quotient spaces $G/H_I$ are exactly the generalized flag manifolds. Indeed, any parabolic subalgebra is given by
\begin{align*}
\mathfrak{p}_I=\mathfrak{t}^\mathbb{C}\oplus\sum_{\alpha\in\Delta^+}\mathfrak{g}_\alpha\oplus\sum_{\alpha\in\Delta_I}\mathfrak{g}_{-\alpha}
\end{align*}
for some subset $I$ of simple roots as before and hence, $\mathfrak{h}_I=\mathfrak{p}_I\cap\mathfrak{g}$. On the other hand, generalized flag manifolds are homogeneous spaces $G^\mathbb{C}/P_I$ where $P_I$ is a parabolic subgroup with Lie algebra~$\mathfrak{p}_I$. The compact real form $G$ acts transitively on $G^\mathbb{C}/P_I$ and we can see that
\begin{align*}
G^\mathbb{C}/P_I=G/P_I\cap G.
\end{align*}
Now $H_I$ is the connected component of $P_I\cap G$ and as generalized flags are simply connected, $G/H_I=G^\mathbb{C}/P_I$.
\end{Remark}

\subsection*{Acknowledgements}

I~like to thank the many people who have shown their interest in this project. In particular, I~am grateful to Oliver Goertsches and Anna Fino for their help on various parts of this work.
I~thank the anonymous referees for their detailed comments, which have greatly improved the quality of this manuscript.

\pdfbookmark[1]{References}{ref}
\LastPageEnding


\begin{thebibliography}{99}
\footnotesize\itemsep=0pt

\bibitem{ADS21}
Agricola I., Dileo G., Stecker L., Homogeneous non-degenerate
  3-{$(\alpha,\delta)$}-{S}asaki manifolds and submersions over quaternionic
  {K}\"ahler spaces,
  \href{https://doi.org/10.1007/s10455-021-09762-9}{\textit{Ann. Global Anal.
  Geom.}} \textbf{60} (2021), 111--141,
  \href{https://arxiv.org/abs/2011.13434}{arXiv:2011.13434}.

\bibitem{ADS23}
Agricola I., Dileo G., Stecker L., Curvature properties of
  3-{$(\alpha,\delta)$}-{S}asaki manifolds,
  \href{https://doi.org/10.1007/s10231-023-01310-5}{\textit{Ann. Mat. Pura
  Appl.}} \textbf{202} (2023), 2007--2033,
  \href{https://arxiv.org/abs/2206.05150}{arXiv:2206.05150}.

\bibitem{AFF}
Agricola I., Ferreira A.C., Friedrich T., The classification of naturally
  reductive homogeneous spaces in dimensions {$n\leq 6$},
  \href{https://doi.org/10.1016/j.difgeo.2014.11.005}{\textit{Differential
  Geom. Appl.}} \textbf{39} (2015), 59--92,
  \href{https://arxiv.org/abs/1407.4936}{arXiv:1407.4936}.

\bibitem{AFS}
Alexandrov B., Friedrich T., Schoemann N., Almost {H}ermitian 6-manifolds
  revisited, \href{https://doi.org/10.1016/j.geomphys.2004.04.009}{\textit{J.
  Geom. Phys.}} \textbf{53} (2005), 1--30,
  \href{https://arxiv.org/abs/math/0403131}{arXiv:math/0403131}.

\bibitem{Rigidity}
Apostolov V., Barbaro G., Lee K.H., Streets J., Rigidity results for
  non-{K}\"ahler {C}alabi--{Y}au geometries on threefolds,
  \href{https://doi.org/10.1007/s00208-025-03308-1}{\textit{Math. Ann.}}
  \textbf{393} (2025), 3609--3637,
  \href{https://arxiv.org/abs/2408.09648}{arXiv:2408.09648}.

\bibitem{ALL}
Apostolov V., Lahdili A., Lee K.H., Pluriclosed 3-folds with vanishing {B}ismut
  {R}icci form: {G}eneral theory in the quasi-regular case,
  \href{https://arxiv.org/abs/2601.04937}{arXiv:2601.04937}.

\bibitem{BFG_BTP}
Brienza B., Fino A., Grantcharov G., C{YT} and {SKT} manifolds with parallel
  {B}ismut torsion, \href{https://doi.org/10.1017/prm.2024.115}{\textit{Proc.
  Roy. Soc. Edinburgh Sect.~A}} \textbf{156} (2026), 901--926,
  \href{https://arxiv.org/abs/2401.07800}{arXiv:2401.07800}.

\bibitem{FinosHKT}
Brienza B., Fino A., Grantcharov G., Verbitsky M., On the structure of compact
  strong {HKT} manifolds,
  \href{https://doi.org/10.1007/s00220-026-05627-z}{\textit{Comm. Math. Phys.}}
  \textbf{407} (2026), 122, 41~pages,
  \href{https://arxiv.org/abs/2505.06058}{arXiv:2505.06058}.

\bibitem{CMS}
Cleyton R., Moroianu A., Semmelmann U., Metric connections with parallel
  skew-symmetric torsion,
  \href{https://doi.org/10.1016/j.aim.2020.107519}{\textit{Adv. Math.}}
  \textbf{378} (2021), 107519, 50~pages,
  \href{https://arxiv.org/abs/1807.00191}{arXiv:1807.00191}.

\bibitem{CortesHasegawa}
Cort\'es V., Hasegawa K., The {H}/{Q}-correspondence and a generalization of
  the supergravity c-map,
  \href{https://doi.org/10.2748/tmj.20221125}{\textit{Tohoku Math.~J.}}
  \textbf{76} (2024), 255--292,
  \href{https://arxiv.org/abs/2207.09013}{arXiv:2207.09013}.

\bibitem{FGSKT}
Fino A., Grantcharov G., C{YT} and {SKT} metrics on compact semi-simple {L}ie
  groups, \href{https://doi.org/10.3842/SIGMA.2023.028}{\textit{SIGMA}}
  \textbf{19} (2023), 028, 15~pages,
  \href{https://arxiv.org/abs/2212.07915}{arXiv:2212.07915}.

\bibitem{FGsurvey}
Fino A., Grantcharov G., A survey on pluriclosed and {CYT} metrics,
  \href{https://doi.org/10.55630/serdica.2024.50.103-124}{\textit{Serdica
  Math.~J.}} \textbf{50} (2024), 103--124.

\bibitem{FLMM_G2}
Fino A., Mart\'{\i}n-Merch\'{a}n L., Raffero A., The twisted {${\rm G}_2$}
  equation for strong {${\rm G}_2$}-structures with torsion,
  \href{https://doi.org/10.4310/pamq.241112235423}{\textit{Pure Appl.
  Math.~Q.}} \textbf{20} (2024), 2711--2767,
  \href{https://arxiv.org/abs/2306.07128}{arXiv:2306.07128}.

\bibitem{FiVe}
Fino A., Vezzoni L., Special {H}ermitian metrics on compact solvmanifolds,
  \href{https://doi.org/10.1016/j.geomphys.2014.12.010}{\textit{J. Geom.
  Phys.}} \textbf{91} (2015), 40--53,
  \href{https://arxiv.org/abs/1502.03476}{arXiv:1502.03476}.

\bibitem{FriG2}
Friedrich T., {${\rm G}_2$}-manifolds with parallel characteristic torsion,
  \href{https://doi.org/10.1016/j.difgeo.2007.06.010}{\textit{Differential
  Geom. Appl.}} \textbf{25} (2007), 632--648,
  \href{https://arxiv.org/abs/math/0604441}{arXiv:math/0604441}.

\bibitem{FrIv}
Friedrich T., Ivanov S., Parallel spinors and connections with skew-symmetric
  torsion in string theory,
  \href{https://doi.org/10.4310/AJM.2002.v6.n2.a5}{\textit{Asian~J. Math.}}
  \textbf{6} (2002), 303--335,
  \href{https://arxiv.org/abs/math/0102142}{arXiv:math/0102142}.

\bibitem{G2heterotic}
Galdeano M., Stecker L., The heterotic {${\rm G}_{2}$} system with reducible
  characteristic holonomy,
  \href{https://doi.org/10.1016/j.geomphys.2025.105635}{\textit{J.~Geom.
  Phys.}} \textbf{217} (2025), 105635, 33~pages,
  \href{https://arxiv.org/abs/2403.00084}{arXiv:2403.00084}.

\bibitem{MGF}
Garcia-Fernandez M., Gonzalez~Molina R., Streets J., Pluriclosed flow and the
  {H}ull--{S}trominger system,
  \href{https://doi.org/10.1016/j.aim.2025.110699}{\textit{Adv. Math.}}
  \textbf{485} (2026), 110699, 95~pages,
  \href{https://arxiv.org/abs/2408.11674}{arXiv:2408.11674}.

\bibitem{MGFJS_pluri}
Garcia-Fernandez M., Jordan J., Streets J., Non-{K}\"ahler {C}alabi--{Y}au
  geometry and pluriclosed flow,
  \href{https://doi.org/10.1016/j.matpur.2023.07.002}{\textit{J.~Math. Pures
  Appl.}} \textbf{177} (2023), 329--367,
  \href{https://arxiv.org/abs/2106.13716}{arXiv:2106.13716}.

\bibitem{IvanovStan}
Ivanov S., Stanchev N., The {R}iemannian {B}ianchi identities of metric
  connections with skew torsion and generalized {R}icci solitons,
  \href{https://doi.org/10.1007/s00025-024-02302-4}{\textit{Results Math.}}
  \textbf{79} (2024), 270, 20~pages,
  \href{https://arxiv.org/abs/2307.03986}{arXiv:2307.03986}.

\bibitem{KSstring}
Kennon A., Streets J., The canonical symmetry reduction of string backgrounds,
  \href{https://arxiv.org/abs/2511.20773}{arXiv:2511.20773}.

\bibitem{KobNom2}
Kobayashi S., Nomizu K., Foundations of differential geometry. {V}ol.~{II},
  \textit{Interscience Tracts Pure Appl. Math.}, Vol.~15, Interscience
  Publishers, New York, 1969.

\bibitem{LauMon}
Lauret J., Montedoro F., Pluriclosed metrics on compact semisimple {Lie}
  groups, \href{https://arxiv.org/abs/2506.21725}{arXiv:2506.21725}.

\bibitem{MorSch}
Moroianu A., Schwahn P., Submersion constructions for geometries with parallel
  skew torsion, \href{https://arxiv.org/abs/2409.14421}{arXiv:2409.14421}.

\bibitem{Obata}
Obata M., Affine connections on manifolds with almost complex, quaternion or
  {H}ermitian structure,
  \href{https://doi.org/10.4099/jjm1924.26.0_43}{\textit{Jpn.~J. Math.}}
  \textbf{26} (1956), 43--77.

\bibitem{OVFinoVezConj}
Ornea L., Verbitsky M., Balanced metrics and {G}auduchon cone of locally
  conformally {K}\"ahler manifolds,
  \href{https://doi.org/10.1093/imrn/rnaf014}{\textit{Int. Math. Res. Not.}}
  \textbf{2025} (2025), rnaf014, 10~pages,
  \href{https://arxiv.org/abs/2407.04623}{arXiv:2407.04623}.

\bibitem{Papadopoulos}
Papadopoulos G., On the rigidity of special and exceptional geometries with
  torsion a closed 3-form,
  \href{https://arxiv.org/abs/2511.20568}{arXiv:2511.20568}.

\bibitem{PPSHypercomplex}
Pedersen H., Poon Y.S., Swann A.F., Hypercomplex structures associated to
  quaternionic manifolds,
  \href{https://doi.org/10.1016/S0926-2245(98)00026-6}{\textit{Differential
  Geom. Appl.}} \textbf{9} (1998), 273--292.

\bibitem{Pham}
Pham D.N., A family of left-invariant {SKT} metrics on the exceptional {L}ie
  group~{${\rm G}_2$}, \textit{Acta Math. Univ. Comenian.~(N.S.)} \textbf{94}
  (2025), 21--46, \href{https://arxiv.org/abs/2407.10037}{arXiv:2407.10037}.

\bibitem{Pittie}
Pittie H.V., The {D}olbeault-cohomology ring of a compact, even-dimensional
  {L}ie group, \href{https://doi.org/10.1007/BF02863632}{\textit{Proc. Indian
  Acad. Sci. Math. Sci.}} \textbf{98} (1988), 117--152.

\bibitem{nKpaper}
Stecker L., Canonical submersions in nearly {K}\"ahler geometry,
  \href{https://doi.org/10.1515/coma-2025-0016}{\textit{Complex Manifolds}}
  \textbf{12} (2025), 20250016, 11~pages,
  \href{https://arxiv.org/abs/2211.14012}{arXiv:2211.14012}.

\bibitem{ZZ_BTP}
Zhao Q., Zheng F., Curvature characterization of {H}ermitian manifolds with
  {B}ismut parallel torsion,
  \href{https://arxiv.org/abs/2407.10497}{arXiv:2407.10497}.

\end{thebibliography}
\end{document}